\documentclass{article}
\usepackage{latexsym}
\usepackage{mathrsfs}
\usepackage{indentfirst}
\usepackage{amsmath,amsthm, amssymb}
\usepackage[dvips]{graphicx}
\usepackage{graphicx}
\usepackage{epsfig}
\usepackage{verbatim}
\usepackage{color}

\usepackage{caption}

\usepackage{enumitem}

\usepackage{epstopdf}

\usepackage{thmtools}

\newtheorem{theorem}{Theorem}[section]
\newtheorem{lemma}[theorem]{Lemma}

\newtheorem{problem}{Problem}

\begin{document}
\title{Path Extendable Tournaments
}
\author{
Zan-Bo Zhang$^{1,2}$\thanks{Email address: eltonzhang2001@gmail.com, zanbozhang@gdufe.edu.cn. Partially supported by
Guangdong Basic and Applied Basic Research Foundation (2020B1515310009), and NSFC (U1811461).}
Weihua He$^{3}$\thanks{Email address: hwh12@gdut.edu.cn.
Supported by Guangdong Basic and Applied Basic Research Foundation (2021A1515012047),
and Science and Technology Program of Guangzhou (202002030401).}
Hajo Broersma$^4$\thanks{Email address: h.j.broersma@utwente.nl.}
Xiaoyan Zhang$^{5}$\thanks{Corresponding Author. Email address: royxyzhang@gmail.com, zhangxiaoyan@njnu.edu.cn.
Supported by NSFC (11871280) and QingLan Project.}
\\ \small $^1$ School of Statistics and Mathematics, Guangdong University of
\\ \small  Finance and Economics, Guangzhou 510320, China
\\ \small $^2$ Institute of Artificial Intelligence and Deep Learning, Guangdong University
\\ \small of Finance and Economics, Guangzhou 510320, China
\\ \small $^3$ School of Mathematics and Statistics, Guangdong University
\\ \small of Technology, Guangzhou 510006, China
\\ \small $^4$ Faculty of Electrical Engineering, Mathematics and Computer Science,
\\ \small University of Twente, P.O. Box 217, 7500 AE Enschede, The Netherlands
\\ \small $^5$Institute of Mathematics \& School of Mathematical Science,
\\ \small  Nanjing Normal University, Nanjing, 210023, China
}
\date{}
\maketitle

\begin{abstract}
\noindent
A digraph $D$ is called \emph{path extendable} if for every nonhamiltonian (directed) path
$P$ in $D$, there exists another path $P^\prime$ with the same initial and
terminal vertices as $P$, and $V(P^\prime) = V (P)\cup \{w\}$ for a vertex $w \in V(D)\setminus V(P)$.
Hence, path extendability implies paths of continuous lengths between every vertex pair.
In earlier works of C. Thomassen and K. Zhang,
it was shown that the condition of small $i(T)$ or positive $\pi_2(T)$
implies paths of continuous lengths between every vertex pair in a tournament $T$,
where $i(T)$ is the irregularity of $T$ and
$\pi_2(T)$ denotes for the minimum number of paths of length $2$ from $u$ to $v$
among all vertex pairs $\{u,v\}$.
Motivated by these results, we study
sufficient conditions in terms of $i(T)$ and $\pi_2(T)$
that guarantee a tournament $T$ is path extendable.
We prove that (1) a tournament $T$ is path extendable
if $i(T)< 2\pi_2(T)-(|T|+8)/6$, and (2) a tournament $T$ is path extendable
if $\pi_2(T) > (7|T|-10)/36$.
As an application, we deduce that almost all random tournaments are path extendable.

$\newline$\noindent\textbf{Key words}: path, path extendability, tournament,
random tournament
\end{abstract}

\section{Introduction, terminology and notation}

%
%

\noindent
It is commonly known that tournaments have a rich path and cycle structure. As an example,
every tournament has a hamiltonian path, and every strong tournament is hamiltonian, vertex-pancyclic,
and even cycle extendable
(for cycle extendability a well-characterized class of exceptional graphs has to be excluded, see \cite{Hendry1989,Moon1969}).

However, many stronger properties involving connecting paths of specific
lengths, such as hamiltonian-connectedness, 
panconnectedness and path extendability,
do not generally hold in tournaments.
Since the late 1960s, many researchers have devoted themselves to finding sufficient conditions
for these properties.


Throughout the paper, we use $n$ or $|D|$ to denote the number of vertices of a digraph $D$.
Let $D$ be a digraph on $n\ge 2$ vertices, and let $u,v\in V(D)$ be two distinct vertices.
We use the term $k$-path to indicate a directed path of length $k$,
i.e., containing $k$ arcs,
and $(u,v)$-$k$-path as shorthand for a $k$-path from $u$ to $v$.
Adopting the terminology of Thomassen \cite{Thomassen1980},
a digraph $D$ is said to be \emph{weakly hamiltonian-connected}
if for any distinct $x,y \in V(D)$ there exists a hamiltonian path from $x$ to $y$ or from $y$ to $x$,
and $D$ is called \emph{strongly hamiltonian-connected}
if for any distinct $x,y \in V(D)$ there exists a hamiltonian path from $x$ to $y$ and from $y$ to $x$.
Similarly, in \cite{Thomassen1980}, $D$ is called \emph{weakly panconnected} if for any distinct $x,y \in V(D)$ there exists  an $(x,y)$-$k$-path or a $(y,x)$-$k$-path for every integer $k$ with $3\le k \le n-1$, and
\emph{strongly panconnected}
if for any distinct $x,y \in V(D)$ there exists  an $(x,y)$-$k$-path and a $(y,x)$-$k$-path for every integer $k$ with $3\le k \le n-1$.

In \cite{Thomassen1980}, Thomassen characterized all weakly hamiltonian-connected tournaments,
and showed that for a tournament being weakly hamiltonian-connected is equivalent to being weakly
panconnected. He also proved that being $4$-strong is a sufficient condition for a
tournament to be strongly hamiltonian-connected.
Here a digraph $D$ is called \emph{$k$-strong} if $D-A$ is strongly connected for any set $A\subseteq V(D)$ with at most $k-1$ vertices.

Then, strengthening a result of Alspach et al., 
he showed that a tournament $T$ that is close to regular, or to be precise, with small \emph{irregularity}
$i(T)\triangleq \max \{|d^+(u)-d^-(u)|\mid u\in V(T)\}$,
is strongly panconnected. 
Note that in a tournament the condition of irregularity can be turned into a degree condition as
\begin{equation} \label{eqn:i(T)andDegree}
\min\{\delta^+(T), \delta^-(T)\} =\frac{n-i(T)-1}{2}.
\end{equation}


\begin{theorem} \label{thm:irregularitystronglypanconnected} (Thomassen \cite{Thomassen1980})
If $T$ is a tournament on $n \geq 5k + 21$ vertices with $i(T) \leq k$, then $T$ is strongly panconnected.
\end{theorem}

Note that in the definition of strongly panconnectedness the existence of paths of length 2 is not required.
A more restrictive concept which includes the existence of $2$-paths was studied by Zhang in \cite{Zhang1982}. There he defines a digraph $D$ to be \emph{completely strong path-connected}
if there exists a $(u,v)$-$k$-path and a $(v,u)$-$k$-path in $D$ for every distinct $u\in V(D)$ and $v\in V(D)$
and every integer $k$ with $2\le k \le n-1$.

Let $T$ be a tournament. We denote %
the number of distinct $(u,v)$-$2$-paths in $T$ by $p_2(u,v)$,
and we let $\pi_2(T)=\min \{p_2(u,v)\mid u,v \in V(T)\}$ be the global infimum of $p_2(u,v)$,
where $u$ and $v$ are distinct vertices of $T$.
In \cite{Zhang1982}, Zhang presented the following characterization of completely strong path-connected tournaments in terms of $\pi_2(T)$.

\begin{theorem} \label{thm:2pathallpaths} (Zhang \cite{Zhang1982})
A tournament $T$ on $n$ vertices is completely strong path-connected if and only if $\pi_2(T)\ge 1$,
unless $T$ belongs to one well-characterized class of exceptional graphs.
\end{theorem}

Motivated by the aforementioned structural results on tournaments and the results due to Hendry on path extendability in undirected graphs \cite{Hendry1990b}, we focus on the analogous concept of path extendability in tournaments.
For this purpose, we say that a
nonhamiltonian path $P$ in a digraph $D$ is \emph{extendable},
if there exists another path $P^\prime$ in $D$ with
the same initial vertex and terminal vertex as $P$,
and with $V(P^\prime)=V(P)\cup \{w\}$ for some vertex $w\in V(D)\setminus V(P)$
(note that Faudree and Gy\'{a}rf\'{a}s \cite{FG1996} study another kind of path extendability in tournaments,
which extends a path from one of its endpoints).
Moreover, by including the following specific length restriction,
a digraph $D$ is called \emph{$\{k+\}$-path extendable},
if every nonhamiltonian path of length at least $k$ is extendable in $D$.
And naturally, we use the shorter term \emph{path extendability} for $\{1+\}$-path extendability.
The current authors have studied path extendability in tournaments, and obtained
the following result,
in which the structure of the exceptional graphs will be described in Section \ref{sec:2PathsPathExt}
and be used to prove Theorem \ref{thm:2PathsPathExt}.


\begin{theorem} \label{thm:path3+Extendable} (Zhang et al. 
\cite{ZZBL2017})
A regular tournament $T$ with at least $7$ vertices is $\{2+\}$-path extendable,
unless $T\in \mathcal{T}_1$, $T\in \mathcal{T}_2$ or {$T=T_0$}.
\end{theorem}



Reflecting on Bondy's idea in his meta-conjecture on hamiltonian undirected graphs \cite{Bondy1973}, and considering results of Bondy on pancyclicity \cite{Bondy1971} and Hendry on cycle extendability \cite{Hendry1990a} and path extendability \cite{Hendry1990b}, one is tempted to expect that imposing similar conditions
as in
Theorem~\ref{thm:irregularitystronglypanconnected}
and Theorem~\ref{thm:2pathallpaths} on a tournament would in fact imply that it is path extendable.
With this in mind, the following problems arise quite naturally.


\begin{problem} \label{prb:pi_2}  (Zhang et al. \cite{ZZBL2017}, Problem 2)
Does $\pi_2(T)\ge 1$ imply a tournament $T$ is path extendable?
\end{problem}

\begin{problem} \label{prb:i} (Zhang et al. \cite{ZZBL2017}, Problem 3)
Let $T$ be a tournament on $n$ vertices with $i(T) \le k$.
Does there exist a function $f(k)$, such that $n \ge f(k)$
implies $T$ is path extendable?
\end{problem}

In the next section, the two above problems will be solved in the negative.
However, we derive several sufficient
conditions in terms of $i(T)$ and $\pi_2(T)$ for path extendability, including a condition that
involves both $i(T)$ and $\pi_2(T)$.
We also discuss the sharpness of our new results, and we deduce that almost all tournaments are path extendable.

We explain some notation here.
Let $D$ be a digraph, and $U_1, U_2, U_3 \subseteq V(D)$ be three pairwise disjoint subsets of $D$.
The subdigraph induced by $U_1$ is denoted by $\langle U_1 \rangle$.
The number of arcs in $D$ from $U_1$ to $U_2$ is denoted by $d^+(U_1, U_2)$.
$U_1 \rightarrow U_2$ means that $U_1$ dominates $U_2$,
i.e. every vertex in $U_1$ sends an arc to every vertex in $U_2$.
For convenience, when $U_1=\{u\}$ we write $u$ instead of $\{u\}$, and
we use $U_1\rightarrow U_2\rightarrow U_3$ as an abbreviation for $U_1\rightarrow U_2$ and $U_2\rightarrow U_3$.
Let $P$ be a path and $u,v$ be two vertices on $P$. Then $P[u,v]$ stands for the subpath of $P$ from $u$ to $v$.
The other terms and notation in this paper are in accordance with the monograph \cite{BanGut2008}.

\section{Main results} \label{sec:MainResults}

Before we turn to the solutions of Problem~\ref{prb:pi_2} and Problem~\ref{prb:i},
let us first define what we mean by a \emph{doubly regular tournament}.
A regular tournament $T$ is called doubly regular if for every vertex pair
$\{u,v\}$ in $T$, the number of common out-neighbors of
$u$ and $v$ is the same constant, which we denote by $\lambda$.
Let $T$ be a doubly regular tournament on $n$ vertices, with $\lambda$ as above.
Then simple calculations show that $\lambda = (n-3)/4$, a result that can be found in an early paper due to Reid and Brown \cite{ReiBro1972}.
Therefore, $n=4\lambda+3$, hence $n \equiv 3 \pmod 4$.
Furthermore, one also easily shows that for any $uv \in A(T)$,
the number of common in-neighbors of $u$ and $v$, $p_2(u,v)$,
and $p_2(v,u)$ are fixed to $\lambda$, $\lambda$, and $\lambda + 1$, respectively.
Hence, in a doubly regular tournament $T$ we have $\pi_2(T)=\lambda=(n-3)/4$.
This shows that the upper bound on $\pi_2(T)$ in the below theorem, which will be proved in Section~\ref{sec:Surplus},
is sharp.

\begin{theorem} \label{thm:pi_2Supremum}
In a tournament $T$, $\pi_2(T)\le (n-3)/4$.
\end{theorem}



With the help of the above discussion on doubly regular tournaments, we can now provide negative answers to the questions in the aforementioned two problems.

\medskip
\noindent
\textbf{Solution to Problem \ref{prb:pi_2}.}
Let $\mathcal{T}_3$ be the class of regular tournaments 
whose vertex set can be partitioned as
$V=V_0\cup V_1 \cup V_2$, 
with $|V_i|=4t+3$ for an integer $t\ge 1$. %
Let
the induced subgraph $\langle V_i \rangle$ 
be a doubly regular tournament 
(these exist for infinitely many choices of $t$ by a result in \cite{ReiBro1972}),
and %
let
$V_i \rightarrow V_{i+1}$  for $i\in \{0,1,2\}$, with the subscripts 
taken modulo $3$.
%
%
 %
Consider any vertex pair $\{u,v\}$ with $uv\in A(T)$ in %
a tournament $T\in \mathcal{T}_3$
on $n=12t+9$ vertices.
If $u,v \in V_i$ for any $i\in \{0,1,2\}$,
then all $2$-paths between $u$ and $v$ must be in $\langle V_i \rangle$.
By %
 the
 above discussion on doubly regular tournaments,
there are $t=(n-9)/12$ $(u,v)$-$2$-paths and $t+1=(n+3)/12$ $(v,u)$-$2$-paths in $\langle V_i \rangle$, %
hence
in $T$ as well.
If $u$ and $v$ are in different $V_i$, then 
there are clearly more than $t$
$(u,v)$-$2$-paths and $(v,u)$-$2$-paths in 
$T$. Hence $\pi_2(T)=(n-9)/12$.
But %
obviously
any hamiltonian path of $\langle V_i \rangle$ for $i\in \{0,1,2\}$ is not extendable.
So we have a negative
answer to the question in Problem~\ref{prb:pi_2}.


\medskip
\noindent
\textbf{Solution to Problem \ref{prb:i}.}
By the main result in \cite{ReiBro1972}, there exist doubly regular tournaments of arbitrarily large (but specific) order. Using the above approach and notation,
we can construct a regular tournament $T\in \mathcal{T}_3$ of sufficiently large order $n$, such that each $\langle V_i \rangle$ is isomorphic to
a doubly regular tournament.
Since $i(T)=0$, for whatever nonnegative integer $k$ we choose, we always have $i(T) \le k$.
Moreover, for any function $f$, we can always construct $T$ as above with sufficiently large $n \ge f(k)$.
Recalling that $T$ is not path extendable, this way we obtain concrete evidence for a negative answer to the question in Problem~\ref{prb:i}.
$\newline$

To continue our exposition with some positive results on path extendability, we start by stating the following theorem which appeared in
Chapter 5 of \cite{Zhang2017}.


\begin{theorem} \label{thm:DRTPathExt} (Zhang \cite{Zhang2017}, Theorem 5.17)
All doubly regular tournaments on at least $7$ vertices are path extendable.
\end{theorem}

Recall that a doubly regular tournament $T$ on $n$ vertices satisfies $i(T)=0$ and $\pi_2(T)=(n-3)/4$. A natural question is whether we can relax one or both of the conditions on $i(T)$ and $\pi_2(T)$ in order to obtain a more general result on path extendability of tournaments.
In Section~\ref{sec:2PathsPathExt},
we will study the exceptional graphs of Theorem~\ref{thm:path3+Extendable} carefully,
and prove the following result,
which improves the above 
result by reducing the lower bound on $\pi_2(T)$ to $(n-9)/12$ and considering regular instead of doubly regular tournaments.

\begin{theorem} \label{thm:2PathsPathExt}
Let $T$ be a regular tournament on $n\geq 9$ vertices. If $\pi_2(T)> (n-9)/12$, then $T$ is path extendable.
\end{theorem}

Note that the lower bound on $\pi_2(T)$ %
in the above result is sharp, as shown by the graph class $\mathcal{T}_3$, and recall that $i(T)=0$ for a regular tournament.
Also note that
the negative solution to Problem~\ref{prb:i} eliminates the possibility to %
establish a sufficient condition for path extendability involving $i(T)$ only.

However, by increasing the lower bound on $\pi_2(T)$ in Problem~\ref{prb:pi_2} and Theorem~\ref{thm:2PathsPathExt},
we can obtain the following sufficient condition for path extendability in general tournaments that is based on $\pi_2(T)$ only.


\begin{theorem} \label{thm:pi_2(T)only}
Let $T$ be a tournament on $n$ vertices. If $\pi_2(T) > (7n-10)/36 $, then $T$ is path extendable.
\end{theorem}

Theorem~\ref{thm:pi_2(T)only} can be viewed as an analogue of Theorem~\ref{thm:2pathallpaths} for path extendability, although it does not characterize path extendable tournaments but only provides a sufficient condition in terms of $\pi_2(T)$.




In fact, Theorem~\ref{thm:pi_2(T)only} follows in a rather straightforward way from our next result, involving 
conditions on both $i(T)$ and $\pi_2(T)$.


\begin{theorem} \label{thm:pi_2(T)i(T)}
Let $T$ be a tournament on $n\ge 9$ vertices
with $\pi_2(T)>(n-9)/12$. 
If $i(T)< 2\pi_2(T)-(n+8)/6 $, then $T$ is path extendable.
\end{theorem}

The idea of posing restrictions on both $i(T)$ and $\pi_2(T)$
reflects a trade-off between allowing $T$ to deviate more from regularity by
increasing the upper bound on $i(T)$, and compensating it by 
increasing the lower bound on $\pi_2(T)$, thereby imposing a stronger condition on $\pi_2(T)$ than in (doubly) regular tournaments.
Theorem \ref{thm:pi_2(T)i(T)} 
not only brings in a new 
combination of conditions, but
these conditions also turn out to be convenient in the structural analysis.

We will prove Theorem~\ref{thm:pi_2(T)i(T)} and the below lemma in Section~\ref{sec:pi_2(T)i(T)},
and show how Theorem~\ref{thm:pi_2(T)only} can be obtained from them.

\begin{lemma} \label{lem:i(T)AndPi(T)}
In a tournament $T$ on $n$ vertices, $i(T)\leq n-4\pi_2(T)-3$.
\end{lemma}

\medskip
Finally we show that almost all tournaments are path extendable.
Its proof we present here is a
straightforward application of Theorem \ref{thm:pi_2(T)only}.

In an early paper, Reid and Beineke proved that almost all tournaments are strong, and
they noted
the following stronger result.

\begin{theorem} (Reid and Beineke \cite{ReiBei1978}) \label{thm:T3-Cycle}
In almost all tournaments, every pair of vertices
lies on a $3$-cycle.
\end{theorem}

The conclusion of Theorem \ref{thm:T3-Cycle} is equivalent to
the statement
that
for every arc $uv$, $p_2(v,u)\ge 1$.
%
%
%
We prove that a stronger result holds for the more general class of random oriented graphs.

Fixing $0 \le p \le 1/2$, consider the model $\overrightarrow{G}_{n,p}$
of
all %
random
oriented graphs $D$ on $n$ vertices,
in which
for any vertex pair $\{u,v\}$
with probability $p$ we have the arc $uv$,
with probability $p$ we have the arc $vu$,
and with probability $1-2p$ we have no arc between $u$ and  $v$.
When $p=1/2$, $\overrightarrow{G}_{n,p}$ becomes the model for random tournaments on $n$ vertices.
%
%
%
In Theorem~5.9 of \cite{Zhang2017}, Zhang confirmed that there exists a constant $c=c(p)$
such that almost all $D\in \overrightarrow{G}_{n,p}$
satisfy that $\pi_2(D)\ge c n$.  However, an explicit $c(p)$ is not given there.
In Section~\ref{sec:RandomTPi_2}, we use a Chernoff bound to prove the following theorem,
which
shows that $\pi_2(D)$
is arbitrarily close to its expectation.

\begin{theorem} \label{thm:RandomTPi_2}
For any $0< \varepsilon < p^2$, almost all oriented
graphs $D\in \overrightarrow{G}(n,p)$ satisfy that $\pi_2(D) \ge (p^2-\varepsilon)n$.
\end{theorem}

And in the same section, using Theorem \ref{thm:pi_2(T)only} and Theorem \ref{thm:RandomTPi_2}, we conclude
that random tournaments are
almost surely
path extendable.

\begin{theorem} \label{thm:RandomTPathExt}
Almost all tournaments are path extendable.
\end{theorem}


A key ingredient in this work is the
newly introduced
concept
 of surplus 
 (see Section~\ref{sec:Surplus}). 
 This concept is based on the following considerations.
Note that even in a doubly regular tournament, $p_2(u,v) \ne p_2(v,u)$ for every vertex pair $\{u,v\}$.
And in a general tournament, $p_2(u,v)$
can assume more different values as
 $\{u, v\}$ runs over all vertex pairs.
That is to say, there
might
exist a lot of vertex pairs $\{u,v\}$ 
with
$p_2(u,v) > \pi_2(u,v)$.
It turns out that, in order to estimate the total number of $2$-paths in a tournament $T$, or in part of $T$,
it
is
easier to consider $p_2(u,v)$ and $p_2(v,u)$ together, rather than
to
consider them separately.
Based on this,
we define
the
surplus
of a vertex pair $\{u,v\}$, obtained by subtracting the global lower bound
$2\pi_2(T)$ from $p_2(u,v)+p_2(v,u)$.
The surplus of a vertex set $U$ is then defined to be the sum of the
surpluses
of all vertex pairs
$\{u,v\}$ 
in $U$.
In Section~\ref{sec:Surplus}, we estimate the possible surplus 
of a vertex pair $\{u,v\}$, and
obtain an infimum for the surplus of a vertex set.
Both are used several times in the proofs that follow.

The rest of the paper
is organized as follows.
In Section~\ref{sec:Surplus}, we
formally
introduce the concept of surplus
and prove Theorem~\ref{thm:pi_2Supremum}.
In Section~\ref{sec:2PathsPathExt}, we study path extendability in regular tournaments and prove Theorem~\ref{thm:2PathsPathExt}.
In Section~\ref{sec:pi_2(T)i(T)}, we study conditions on $i(T)$ and $\pi_2(T)$ for path extendability,
and prove Theorem~\ref{thm:pi_2(T)i(T)}, Lemma~\ref{lem:i(T)AndPi(T)} and Theorem~\ref{thm:pi_2(T)only}.
In Section~\ref{sec:RandomTPi_2}, we investigate $2$-paths and path extendability in random oriented graphs, and prove
Theorem~\ref{thm:RandomTPi_2} and Theorem~\ref{thm:RandomTPathExt}.
In Section~\ref{sec:Final}, we conclude the paper with some final discussion.

\section{Surplus and supremum of $\boldsymbol{\pi_2(T)}$; proof of Theorem~\ref{thm:pi_2Supremum}} \label{sec:Surplus}


Let $u$ and $v$ be two distinct vertices in a tournament $T$.
We define the \emph{surplus of $\{u,v\}$} by:
$$s(u,v)=p_2(u,v)+p_2(v,u)-2\pi_2(T).$$

It is obvious that $s(u,v)\ge 0$ and that it measures
the excess 
of $2$-paths between a vertex pair in both directions
with regard to the minimum 
$2\pi_2(T)$.

Suppose $u\rightarrow v$. Then
\begin{equation} \label{eqn:UdominatesV}
p_2(u,v)-p_2(v,u)=d^+(u)-d^+(v)-1.
\end{equation}
This follows from counting the cardinalities of the four different sets of common neighborhoods of $u$ and $v$.
Let $x_1=p_2(u,v)$ denote the number of vertices in $V(T)\setminus\{u,v\}$ that are out-neighbors of $u$ and in-neighbors of $v$,
$x_2$ the number of vertices in $V(T)\setminus\{u,v\}$ that are common out-neighbors of $u$ and $v$,
$x_3$ the number of vertices in $V(T)\setminus\{u,v\}$ that are common in-neighbors of $u$ and $v$,
and $x_4=p_2(v,u)$ the number of vertices in $V(T)\setminus\{u,v\}$ that are in-neighbors of $u$ and out-neighbors of $v$.
Then, $p_2(u,v)-p_2(v,u)=x_1-x_4$, whereas $d^+(u)-d^+(v)=x_1+1-x_4$. Hence, the above equality.

Similarly, if $v\rightarrow u$, we get
\begin{equation} \label{eqn:UdominatedByV}
p_2(u,v)-p_2(v,u)=d^+(u)-d^+(v)+1.
\end{equation}

Now, suppose $d^+(u)\ge d^+(v)$. We discuss three cases.
%

Case 1. $d^+(u)=d^+(v)$. 
Let us further assume that $u\rightarrow v$.
Then, by (\ref{eqn:UdominatesV}), we have $p_2(v,u)=p_2(u,v)+1 \ge \pi_2(T)+1$.
Hence, $s(u,v)\ge 1$.

Case 2. 
%
$d^+(u)=d^+(v)+1$. If $u\rightarrow v$,
then by (\ref{eqn:UdominatesV}), we have $p_2(v,u)=p_2(u,v)$.
Thus if $p_2(u,v)=p_2(v,u)=\pi_2(T)$, we have $s(u,v)=0$.
If $v\rightarrow u$, then by (\ref{eqn:UdominatedByV}), we have
$p_2(v,u)=p_2(u,v)+2 \ge \pi_2(T)+2$, and  $s(u,v)\ge 2$.
%

Case 3. $d^+(u)\ge d^+(v)+2$. If $u\rightarrow v$,
then by (\ref{eqn:UdominatesV}), we have $p_2(u,v)=p_2(v,u)+1 \ge \pi_2(T)+1$, and
 $s(u,v) \ge 1$.
If $v\rightarrow u$, then by (\ref{eqn:UdominatedByV}), we have
$p_2(v,u)=p_2(u,v)+3 \ge \pi_2(T)+3$, and $s(u,v)\ge 3$.
%

Using the above facts,
it is easy to prove the next lemma.

\begin{lemma} \label{lem:VSurplus0}
Let $u$ and $v$ be two distinct vertices in a tournament $T$ with $d^+(u) \ge d^+(v)$.
If $u\rightarrow v$, then $s(u,v)\ge |d^+(u)-d^+(v)-1| \ge 0$.
If $v\rightarrow u$, then $s(u,v)\ge |d^+(u)-d^+(v)+1| \ge 1$.
In particular, if $s(u,v)=0$, then $d^+(u)-d^+(v)=1$ and $u\rightarrow v$.
\end{lemma}

We define
the
\emph{surplus of a vertex set} $W\subseteq V(T)$,  denoted by $s(W)$,
as the sum of the surpluses of all vertex pairs $\{u,v\}$ in $W$.
Thus the number of $2$-paths in $T$ with end-vertices in $W$ is
$$2{|W| \choose 2}\pi_2(T)+s(W).$$

For any $u,v\in W$, by Lemma~\ref{lem:VSurplus0},
only if the out-degree of $u$ and $v$ differ by $1$ it is possible that $s(u,v)=0$.
Thus, to estimate the number of vertex pairs $\{u,v\}$ with $s(u,v)=0$,
we firstly calculate the vertex pairs whose out-degrees differ by $1$.
Let $n_i$ be the order of the set $\{v\mid v\in W, d^+_T(v)=i\}$ for $0\le i \le n-1$.
Then the number of vertex pairs $\{u,v\}$  in $W$ whose out-degrees differ by $1$ is

\begin{equation} \label{eqn:NumberOfSurplus0AtMost}
\sum_{i=0}^{n-2} n_in_{i+1}
\le \sum_{\substack{i=0 \\ i \text{ is odd}}}^{n-1} n_i \sum_{\substack{i=0 \\ i \text{ is even}}}^{n-1} n_{i}
= \sum_{\substack{i=0 \\ i \text{ is odd}}}^{n-1} n_i  (|W|-\sum_{\substack{i=0 \\ i \text{ is odd}}}^{n-1} n_i )
\le \Big \lfloor \frac{|W|}{2}\Big \rfloor \Big \lceil \frac{|W|}{2}\Big \rceil.
 \end{equation}

For any of the other (at least ${|W| \choose 2} - \lfloor |W|/2\rfloor \lceil |W|/2\rceil$)
vertex pairs $\{u, v\}$ whose out-degrees are equal or differ by at least $2$, we have $s(u,v) \ge 1$. Therefore,
\begin{equation} \label{eqn:UpperBoundSurplus}
s(W)=\sum_{\substack{u,v\in W\\ u\ne v}}s(u,v) \ge {|W| \choose 2} - \Big\lfloor \frac{|W|}{2}\Big\rfloor \Big\lceil \frac{|W|}{2}\Big\rceil.
\end{equation}

Now we use (\ref{eqn:UpperBoundSurplus}) to prove Theorem~\ref{thm:pi_2Supremum}.

\begin{proof} [\textit{Proof of Theorem \ref{thm:pi_2Supremum}}]
\noindent
Let $n$ be the number of vertices in $T$.
When $n\le 2$, there is no $2$-path in $T$.
Therefore, we may assume that $n\ge 3$.
We distinguish two cases according to the parity of $n$.

\medskip
\noindent
\textbf{Case 1. $\boldsymbol{n}$ is odd.}

\noindent
Considering any vertex $w\in V(T)$, the number of $2$-paths that go through $w$ is
\begin{equation} \label{eqn:IndegreeoutdegreeOdd} \nonumber
N^+(w)\cdot N^-(w)\le \Big (\frac{n-1}{2}\Big )^2.
\end{equation}
Therefore, the total number of $2$-paths in $T$ is at most $n(n-1)^2/4$. By (\ref{eqn:UpperBoundSurplus}),
\begin{equation} \label{eqn:CalAll2Paths} \nonumber
\begin{split}
 n(n-1)\pi_2(T) + {n \choose 2} - \Big \lfloor \frac{n}{2}\Big \rfloor \Big\lceil \frac{n}{2} \Big\rceil
&= n(n-1)\pi_2(T) + {n \choose 2} - \frac{(n-1)(n+1)}{4} \\
&\le n(n-1)\pi_2(T) +s(V(T)) \\
&= \sum_{\substack{u,v\in V(T) \\ u\neq v}} p_2(u,v) \\
 &\le \frac{n(n-1)^2}{4},
 \end{split}
\end{equation}
from which we get
$$4\pi_2(T) \le n-3+\frac{n+1}{n}=n-2+\frac{1}{n}.$$

Since $1/n <1$, $4\pi_2(T)$ is even and $n$ is odd, we have $4\pi_2(T)\le n-3$ and thus $\pi_2(T)\le (n-3)/4$.

\medskip
\noindent
\textbf{Case 2. $\boldsymbol{n}$ is even.}

\noindent
Similarly to Case 1, considering any vertex $w\in V(T)$, the number of $2$-paths that go through $w$ is
\begin{equation} \label{eqn:IndegreeoutdegreeEven}
N^+(w)\cdot N^-(w)\le \frac{n}{2}\Big (\frac{n}{2}-1 \Big )=\frac{n(n-2)}{4},
\end{equation}
and the total number of $2$-paths in $T$ is at most $n^2(n-2)/4$.
By (\ref{eqn:UpperBoundSurplus}),
the total number of $2$-paths in $T$ is at least $$n(n-1)\pi_2(T)+{n \choose 2} - \Big(\frac{n}{2}\Big)^2.$$
Therefore
\begin{equation} \label{eqn:CallAll2PathsEven}
n(n-1)\pi_2(T)+{n \choose 2} - \Big(\frac{n}{2}\Big)^2
 \le  \sum_{\substack{u,v\in V(T) \\ u\neq v}} p_2(u,v) \le \frac{n^2(n-2)}{4},
\end{equation}
$\noindent$from which we obtain $\pi_2(T)\le (n-2)/4$.

If $\pi_2(T)= (n-2)/4$, then all inequalities in (\ref{eqn:IndegreeoutdegreeEven})
 and (\ref{eqn:CallAll2PathsEven}) are in fact equalities.
By equality in (\ref{eqn:IndegreeoutdegreeEven}),
$T$ must be a tournament with $n/2$ vertices of out-degree $n/2$ and $n/2$ vertices of out-degree $n/2-1$
(such a tournament is called \emph{almost regular}).
And in order that the first equality holds in (\ref{eqn:CallAll2PathsEven}),
for any vertex pair $\{u,v\}$ with $d^+(u)=n/2$ and $d^+(v)=n/2-1$, we must have $s(u,v)=0$.
By Lemma \ref{lem:VSurplus0}, this implies $u\rightarrow v$.
Hence, every vertex of out-degree $n/2$
dominates every vertex of out-degree $n/2-1$, and $T$ is not strong, a contradiction to $\pi_2(T)\ge 1$.
So $\pi_2(T)<(n-2)/4$.
Since $n$ is even, $\pi_2(T)\le (n-4)/4 < (n-3)/4$. This completes the proof of Theorem \ref{thm:pi_2Supremum}.
\end{proof}

Let $\langle W \rangle$ be the subtournament of $T$ induced by $W$.
The below lemma is a structural result on $\langle W \rangle$ when equality holds in (\ref{eqn:UpperBoundSurplus}).

\begin{lemma} \label{lem:Surplus}
Let $W$ be a subset of $V(T)$ for a tournament $T$.
Then $$s(W)\ge {|W| \choose 2} - \Big\lfloor \frac{|W|}{2}\Big\rfloor \Big\lceil \frac{|W|}{2}\Big\rceil.$$
Furthermore, equality holds only when one of the following holds.

(i) $|W|\le (|T|+1)/2$,
$W$ can be partitioned into $W_0$ and $W_1$, such that
$\{|W_0|,|W_1|\}=\{\lfloor |W|/2\rfloor, \lceil |W|/2 \rceil\}$,
$W_0$ consists of vertices of degree $d$ in $T$, $W_1$
consists of vertices of degree $d+1$ in $T$ for a certain integer $d$,
and $W_1 \rightarrow W_0$.

(ii) $|W|\le |T|/3+2$,
$W$ can be partitioned into $W_0$, $W_1$ and $W_2$, such that $|W_0|, |W_1|, |W_2| > 0$,
$\{|W_0|+|W_2|,|W_1|\}=\{\lfloor |W|/2\rfloor, \lceil |W|/2 \rceil\}$,
$W_0$, $W_1$, and $W_2$ consist of vertices of degree $d-1$, $d$ and $d+1$ in $T$
respectively, for a certain integer $d$, and $W_i \rightarrow W_j$ for $i>j$.
\end{lemma}

\begin{proof}
By (\ref{eqn:UpperBoundSurplus}), $s(W)\ge {|W| \choose 2} - \lfloor |W|/2 \rfloor \lceil |W|/2 \rceil.$
For equality to hold, we must have equality in (\ref{eqn:NumberOfSurplus0AtMost}) as well.
In order that all equalities hold in (\ref{eqn:NumberOfSurplus0AtMost}) and (\ref{eqn:UpperBoundSurplus}),
there can be at most three consecutive $n_i$ which are nonzero. Furthermore, the number of vertices of
odd degree and the number of vertices of even degree must be equal or they differ by exactly $1$.

\medskip
\noindent
\textbf{Case 1. There are exactly two consecutive $\boldsymbol{n_i}$ which are nonzero.}

\noindent
Assume that only $n_d$ and $n_{d+1}$ are nonzero.
Let the set of vertices in $W$ with degree $d$ and $d+1$ be $W_0$ and $W_1$, respectively.
Equalities in (\ref{eqn:NumberOfSurplus0AtMost}) and (\ref{eqn:UpperBoundSurplus}) imply that
$n_d=n_{d+1}$ or $|n_d-n_{d+1}|=1$,
that is $\{n_d, n_{d+1}\}=\{\lceil W/2 \rceil, \lfloor W/2 \rfloor \}$.
Furthermore, in order that $s(u,v)=0$ for $u\in W_1$ and $v\in W_0$, we must have $u\rightarrow v$,
hence $W_1 \rightarrow W_0$.

For any arc $u_0u_1$ with $u_0, u_1 \in W_1$ and any $v \in W_0$, we have a $2$-path $u_0u_1v$.
Similarly, for any arc $v_0v_1$ with $v_0, v_1 \in W_0$ and any $u \in W_1$, we have a $2$-path $uv_0v_1$.
Furthermore, in order that equality holds we must have $s(u,v)=0$, and hence $p_2(u,v)=\pi_2(T)$ for
every $u\in W_1$ and $v\in W_0$.
Hence, calculating the total number of $2$-paths from every $u\in W_1$ to every $v\in W_0$, we have

\begin{equation} \nonumber
n_dn_{d+1}\pi_2(T)=\sum_{u\in W_1, v\in W_0} p_2(u,v) \ge \frac{n_{d+1}(n_{d+1}-1)}{2}\cdot n_d + \frac{n_{d}(n_{d}-1)}{2}\cdot n_{d+1}
=n_d n_{d+1}\Big(\frac{|W|}{2}-1\Big).
\end{equation}

So, $\pi_2(T) \ge |W|/2-1$.
However, by Theorem~\ref{thm:pi_2Supremum},
$\pi_2(T) \le (n-3)/4$, so $|W|/2-1 \le (n-3)/4$ and we get $|W| \le (n+1)/2$.
Thus (i) holds.

\medskip
\noindent
\textbf{Case 2. There are three consecutive $\boldsymbol{n_i}$ which are nonzero.}

\noindent
Assume that $n_{d-1}$, $n_d$ and $n_{d+1}$ are nonzero. Let the set of vertices of degree $d-1$, $d$ and $d+1$
be $W_0$, $W_1$ and $W_2$, respectively.
By equalities in (\ref{eqn:NumberOfSurplus0AtMost}),
$n_d=n_{d+1}+n_{d-1}$ or $|n_d-(n_{d+1}+n_{d-1})|=1$.
Hence $\{|W_0|+|W_2|, |W_1|\}=\{\lfloor|W|/2\rfloor, \lceil|W|/2\rceil\}$.
Furthermore, in order that $s(u,v)=0$, for $u\in W_1$ and $v\in W_0$ or $u\in W_2$ and $v\in W_1$,
we must have $u\rightarrow v$, hence $W_2 \rightarrow W_1 \rightarrow W_0$.
Moreover, we need $s(u,v)=1$ for $u\in W_0$ and $v\in W_2$, and by Lemma~\ref{lem:VSurplus0}, $u \rightarrow v$.
So, $W_2\rightarrow W_0$.

For any arc $u_0u_1$ where $u_0,u_1\in W_2$ and $v\in W_0$, we have a $2$-path $u_0u_1v$.
Similarly, for any arc $v_0v_1$ where $v_0,v_1\in W_0$ and $u\in W_2$, we have a $2$-path $uv_0v_1$.
Furthermore, for any $u\in W_2$, $w\in W_1$ and $v\in W_0$, we have a $2$-path $uwv$.
Finally, for every $u\in W_2$ and $v\in W_0$,
$d^+(u)=d^+(v)+2$, and by Lemma \ref{lem:VSurplus0}, $s(u,v)\ge 1$.
In order that equalities in (\ref{eqn:NumberOfSurplus0AtMost}) and (\ref{eqn:UpperBoundSurplus}) hold
we need $s(u,v)=1$, which holds only when $p_2(u,v)=p_2(v,u)+1=\pi_2(T)+1$ .
Therefore, considering the total number of $2$-paths from vertices in $W_2$ to vertices in $W_0$, we have
\begin{equation} \nonumber
\begin{split}
n_{d-1}n_{d+1}(\pi_2(T)+1) &=
\sum_{u\in W_2, v\in W_0} p_2(u,v) \\
&\ge
\frac{n_{d-1}(n_{d-1}-1)}{2}\cdot n_{d+1} + n_{d-1}n_d n_{d+1} + \frac{n_{d+1}(n_{d+1}-1)}{2}\cdot n_{d-1} \\
&= n_{d-1}n_{d+1}\Big(n_d+\frac{n_{d-1}+n_{d+1}}{2}-1\Big) \\
&= n_{d-1}n_{d+1}\Big(n_d+\frac{|W|-n_d}{2}-1\Big) \\
&= n_{d-1}n_{d+1}\Big(\frac{|W|+n_d}{2}-1\Big).
\end{split}
\end{equation}

Therefore
$$\pi_2(T)+1\ge \frac{|W|+ n_d}{2}-1.$$
By Theorem \ref{thm:pi_2Supremum}, $ \pi_2(T) \le (n-3)/4$. Thus we have
$$(n-3)/4+1 \ge (|W|+n_d)/2-1 \ge (|W|+\lfloor |W|/2 \rfloor)/2-1 \ge (|W|+ (|W|-1)/2)/2-1,$$
from which we get
$$|W| \le \frac{n}{3}+2,$$
and (ii) holds. This completes the proof of Lemma~\ref{lem:Surplus}.

\end{proof}

In a regular tournament, the lower bound on $s(W)$ can be improved.

\begin{lemma} \label{lem:SurplusRegular}
Let $W$ be a subset of $V(T)$ for a regular tournament $T$. Then
$$s(W)\ge {|W| \choose 2}.$$
\end{lemma}
\begin{proof}
Since $d^+(u)=d^+(v)$ for every $u,v\in W$, we have $s(u,v)\ge 1$. Therefore,
$$s(W)=\sum_{\substack{u,v\in W\\ u\ne v}}s(u,v) \ge {|W| \choose 2}.$$
\end{proof}

\section{
Regular Tournaments; Proof of Theorem~\ref{thm:2PathsPathExt}} \label{sec:2PathsPathExt}

Recall that Theorem \ref{thm:path3+Extendable} states that
a regular tournament $T$ with at least $7$ vertices is $\{2+\}$-path extendable,
unless $T\in \mathcal{T}_1$, $T\in \mathcal{T}_2$ or {$T=T_0$}.
The current proof is based on the structure of the exceptional graphs $\mathcal{T}_1$, $\mathcal{T}_2$ and {$T_0$},
which is given below.

\begin{figure}[h]
\centering
\includegraphics[width=0.85\linewidth]{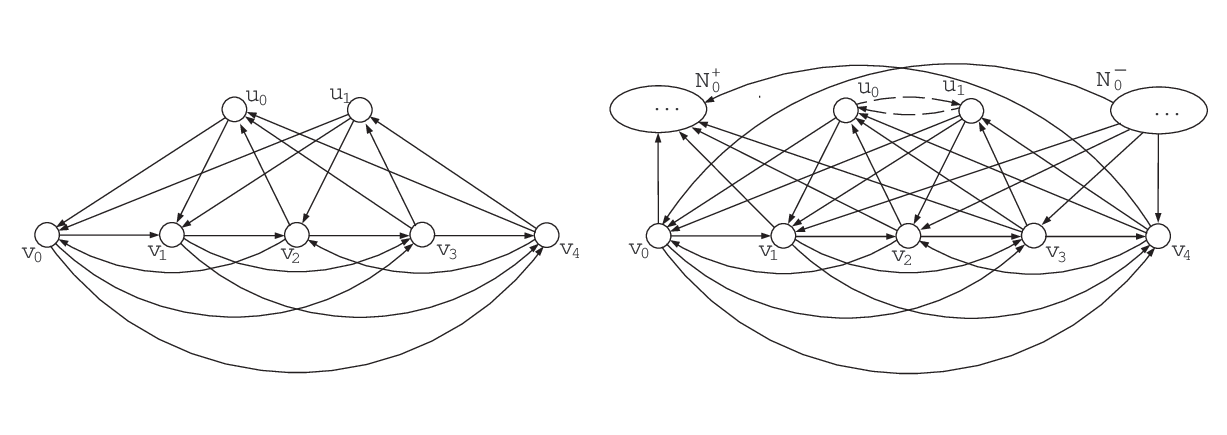}
\caption{The digraph $D_{1}$ (left) and the tournaments in $\mathcal{T}_1$}
\label{figure:TwoHybridVerticesAll}
\end{figure}

Let $D_{1}$ be the digraph at the left in Figure~\ref{figure:TwoHybridVerticesAll}.
Then $\mathcal{T}_1$ is the class of regular tournaments whose vertex set can be partitioned
into four parts $V=\{v_0, v_1, v_2, v_3, v_4\}$, $\{ u_0, u_1\}$, $N_0^+$ and $N_0^-$,
where $\langle V\cup \{u_0, u_1\} \rangle=D_{1}+u_0u_1$ or $D_{1}+u_1u_0$,
$|N_0^+|=|N_0^-|=0$ or $|N_0^+|=|N_0^-|\geq 3$, $V\rightarrow N_0^+$ and $N_0^- \rightarrow V$.
A sketch of the members of the class $\mathcal{T}_1$ is depicted at the right in Figure~\ref{figure:TwoHybridVerticesAll}.
In the sketch there are missing arcs that should be added in order to turn the digraphs into regular tournaments. This might be done in an arbitrary manner, as long as one respects the constraint that
 the resulting tournament is regular.

\begin{figure}[h]
\centering
\includegraphics[width=0.5\linewidth]{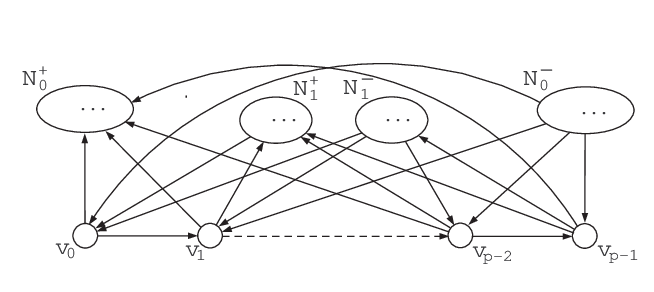}
\caption{The tournaments in $\mathcal{T}_2$}
\label{figure:TwoHybridSets}
\end{figure}

Let $p\ge 3$ be an odd integer, and let $\mathcal{T}_2$ be the set of regular tournaments whose vertex set can be partitioned into five sets $V=\{v_i, 0\leq i \leq p-1\}$, $N_0^+$, $N_1^+$, $N_0^-$ and $N_1^-$,
such that $ |N_0^+|=|N_0^-| = n_0$, $|N_1^+|=|N_1^-|=n_1$, for $n_1\leq (p-1)/2$ and $n_0+n_1 \geq p$, and let $v_0v_1\ldots v_{p-1}$ be a path. Adopting the above convention, let $V\rightarrow N_0^+$, $N_0^-\rightarrow V$, $V\setminus \{v_0\} \rightarrow N_1^+$, $N_1^+\rightarrow v_0$, $N_1^- \rightarrow V\setminus \{v_{p-1}\} $ and $v_{p-1} \rightarrow N_1^-$.
See Figure~\ref{figure:TwoHybridSets} for a sketch of the members of $\mathcal{T}_2$.
As in Figure ~\ref{figure:TwoHybridVerticesAll}, the missing arcs should be added in one way or another in order to turn the digraphs into regular tournaments.

\begin{figure}[h]
\centering
\includegraphics[width=0.27\linewidth]{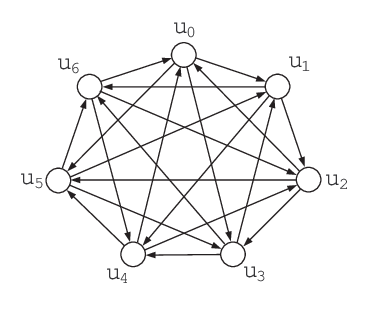}
\caption{The tournament $T_0$}
\label{figure:Seven-Exception}
\end{figure}



Let $T_0$ be the regular tournament on seven vertices depicted in Figure~\ref{figure:Seven-Exception}.

Let $T$ be a regular tournament with $n\ge 9$ vertices and $\pi_2(T)> (n-9)/12 \ge 0$. This
 implies that every path of length $1$ is extendable.
Moreover, by Theorem~\ref{thm:path3+Extendable}, either $T$ is $\{2+\}$-path extendable,
or $T$ belongs to the exceptional graph class $\mathcal{T}_1\cup \mathcal{T}_2 \cup \{T_0\}$.
If $T$ is $\{2+\}$-path extendable, then together with $\pi_2(T)\ge 1$, $T$ is path extendable.
Therefore, we only need to discuss the case that $T$ belongs to the exceptional graphs.

Since $n\ge 9$, $T\ne T_0$.
If $T \in \mathcal{T}_1$, then $T$ contains a path $P=v_0v_1v_2v_3v_4$ as illustrated in Figure~\ref{figure:TwoHybridVerticesAll}.
However, by the definition of $T$ and as shown in Figure~\ref{figure:TwoHybridVerticesAll},
 $\{v_0, v_1\} \rightarrow N_0^+ \cup \{v_3, v_4\}$ ,
 $N_1^+ \rightarrow \{v_0, v_1\}$ and $v_1 \rightarrow v_2 \rightarrow v_0$,
 thus $p_2(v_0,v_1)=0$, contradicting $\pi_2(T)> 0$.

So we may assume that $T\in \mathcal{T}_2$. Consider the path $P=v_0v_1\ldots v_{p-1}$ in $T$.
If $p=3$, then there is no $2$-path from $v_0$ to $v_1$, contradicting $\pi_2(T)> 1$.
Thus we may further assume that $p\ge 5$.

Let $P^\prime=P[v_1, v_{p-2}]$. Then $|V(P^\prime)|=p-2\ge 3$.
Since $V(P^\prime)$ dominates $N_0^+\cup N_1^+$ and is dominated by $N_0^-\cup N_1^-$,
any $2$-paths between vertices in $V(P^\prime)$ must either be in $\langle V(P^\prime) \rangle$ or
go through $v_0$ or $v_1$.
For every $v_i\in V(P^\prime)$, the number of $2$-paths that go through $v_i$
in $\langle V(P^\prime)\rangle$ is at most  $((p-3)/2)^2$.
Thus the total number of $2$-paths in $P^\prime$ is at most $(p-2)((p-3)/2)^2$.
Since $p-2$ is odd, the number of $2$-paths between vertices in $V(P^\prime)$
that go through $v_0$ or $v_{p-1}$ is at most $2((p-1)/2)((p-3)/2)=(p-1)(p-3)/2$.
So, the total number of $2$-paths between vertices in $V(P^\prime)$ is at most

\begin{equation} \nonumber
(p-2)\Big(\frac{p-3}{2}\Big)^2+\frac{(p-1)(p-3)}{2}.
\end{equation}

Since $T$ is regular, by Lemma~\ref{lem:SurplusRegular},
$s(V(P^\prime))\ge {p-2 \choose 2}$, and hence
the total number of $2$-paths among vertices in $V(P^\prime)$ is at least
$(p-2)(p-3)\pi_2(T)+{p-2 \choose 2}$. So, we have
\begin{equation} \nonumber
(p-2)(p-3)\pi_2(T)+{p-2 \choose 2} \le  (p-2)\Big(\frac{p-3}{2}\Big)^2+\frac{(p-1)(p-3)}{2},
\end{equation}
$\noindent$by which we get
$$4\pi_2(T)\le p-3+\frac{2}{p-2}.$$

Since $p\ge 5$, $2/(p-2) \le 1$, and since $\pi_2(T)$ and $p$ are integers, $4\pi_2(T)\le p-3.$
Furthermore, by $n_0+n_1\ge p$ in the definition of $T\in \mathcal{T}_2$, we have $p\le n/3$. Thus, we obtain
$$\pi_2(T)\le \frac{1}{4}(\frac{n}{3}-3)=\frac{n-9}{12},$$
contradicting $\pi_2(T)>(n-9)/12$.

Concluding, we have shown that under the condition $\pi_2(T)>(n-9)/12$, $T$ cannot
be in the class of exceptional graphs, and thereby completed the proof of Theorem~\ref{thm:2PathsPathExt}.

\section{Condition on $\boldsymbol{i(T)}$ and $\boldsymbol{\pi_2(T)}$; 
Proof of Theorem~\ref{thm:pi_2(T)i(T)} and Theorem~\ref{thm:pi_2(T)only}}
\label{sec:pi_2(T)i(T)}

\subsection{Proof of Theorem~\ref{thm:pi_2(T)i(T)}}

Let $P = u_0 u_1 \ldots u_{p-1}$ be a non-extendable path on $p$ vertices in a
tournament $T$, and let $w \in V(T)\setminus V(P)$.
We classify $w$ into three categories. If $w\rightarrow V(P)$, we say that $w$ is a \emph{dominating vertex} of $P$.
If $V(P)\rightarrow w$, we say that $w$ is a \emph{dominated vertex} of $P$. In the other case, $w$ is called a \emph{hybrid
vertex} of $P$. We denote the set of hybrid vertices of $P$ by $H(P)$, and we let $|H(P)|=h(P)$.

Supposing $w$ is a hybrid vertex of $P$, there cannot exist $0\le i < j \le p-1$ such that
$u_i \rightarrow w \rightarrow u_j$, or there must be an integer $i^\prime$ such that $i\le i^\prime \le j-1$ and
$u_{i^\prime} \rightarrow w \rightarrow u_{i^\prime+1}$,
and then $P$ can be extended to $u_0\ldots u_{i^\prime} w u_{i^\prime+1} \ldots u_{p-1}$,
a contradiction. However, there must be arcs from $w$ to $P$ and arcs from $P$ to $w$.
Therefore, there exists $0\le k \le q-2$ such that
$\{u_{k+1}, \ldots, u_{p-1}\} \rightarrow w \rightarrow \{u_0, \ldots, u_k\}$.
We say that $w$ \emph{switches at} $k$.

The concepts of dominating vertex, dominated vertex and hybrid vertex are applicable to any subpath $P_0$ of $P$.
A dominating (dominated) vertex of $P$ remains dominating (dominated) for $P_0$. A hybrid vertex of $P$ may
become dominating, dominated or stay hybrid for $P_0$.
In our proof, we distinguish the cases that $P^\prime=u_1u_2\ldots u_{p-2}$ has a hybrid vertex or not, or
equivalently, whether there exists a vertex $w\in V(T)\setminus V(P)$
which switches at $k_0$ for some $1\le k_0 \le p-3$.
We have the following result when $P^\prime$ has any hybrid vertex.

\begin{lemma} \label{lem:h(P)}
Let $P = u_0 u_1 \ldots u_{p-1}$ be a non-extendable path
on $p\ge 4$ vertices in a
tournament $T$.
Let $P^\prime=u_1u_2\ldots u_{p-2}$.
If $P^\prime$ has a hybrid vertex,
then
$h(P)\leq i(T)+2$.
\end{lemma}

\begin{proof}
Suppose $P^\prime$ has a hybrid vertex.
Then there is at least one hybrid vertex $w$ switching at $1\leq k_0 \leq p-3$,
that is $u_{k_0+1}\rightarrow w \rightarrow u_{k_0}$.
We have $u_j \rightarrow w \rightarrow u_i$, for all $i \in [1, k_0]$ and $j \in [k_0+1, p-2]$. In particular,
$u_{p-2} \rightarrow w \rightarrow u_1$. For any $1\leq i \leq p-3$, if $u_0u_{i+1} \in A(T)$ and
$u_i u_{p-1} \in A(T)$, then $P$ can be extended to
$Q=u_0u_{i+1}\ldots u_{p-2}wu_1 \ldots u_i u_{p-1}$,
contradicting that $P$ is not extendable. Hence $|\{u_0u_{i+1}, u_iu_{p-1}\} \cap A(T) |\leq 1$, and
\begin{equation} \nonumber \label{DegreeonP}
d_P^+(u_0)+d_P^-(u_{p-1}) \leq p-3+1+1+2 = p+1.
\end{equation}

Let $F=T-V(P)$. Then,
\begin{equation} \label{eqn:DegreetoF}
\begin{split}
d^+_F(u_0) + d^-_F(u_{p-1})
&=d^+(u_0)+d^-(u_{p-1})-d_p^+(u_0)-d_p^-(u_{p-1}) \\
&\geq \frac{n-1-i(T)}{2}+\frac{n-1-i(T)}{2}-(p+1) \\
&=n-p-i(T)-2.
\end{split}
\end{equation}

For any $w\in H(P)$, we have $u_{p-1}\rightarrow w \rightarrow u_0$,
and hence $H(P)\subseteq N_F^-(u_0)\cap N_F^+(u_{p-1})$. Since $T$ is a tournament,
\begin{equation}\label{eqn:Completeness}N_F^-(u_0)\cap N_F^+(u_{p-1}) = V(F)\setminus (N_F^+(u_0) \cup N_F^-(u_{p-1})).\end{equation}

Furthermore, since $P$ is not extendable,
\begin{equation}\label{eqn:Nonintersect} N_F^+(u_0) \cap N_F^-(u_{p-1}) = \emptyset.\end{equation}

By (\ref{eqn:DegreetoF}), (\ref{eqn:Completeness}) and (\ref{eqn:Nonintersect}),
\begin{equation} \nonumber
\begin{split}
  h(P)
 &=|H(P)| \\
 &= |N_F^-(u_0)\cap N_F^+(u_{p-1})|  \\
 &= |V(F)\setminus (N_F^+(u_0) \cup N_F^-(u_{p-1}))| \\
 &= |F|- |N_F^+(u_0)| - |N_F^-(u_{p-1}))|\\
 &= |F|- d_F^+(u_0) - d_F^-(u_{p-1})  \\
 &\leq |F|-(n-p-i(T)-2)  \\
 &= i(T)+2,
\end{split}
\end{equation}
$\noindent$which completes the proof of the lemma.
\end{proof}

In the proof of Theorem~\ref{thm:pi_2(T)i(T)}, we will need the following theorem which provides a lower bound on the length of $P$.
It has been proved in \cite{Zhang2017}. For the sake of completeness, we include the proof here.
\begin{theorem} (Zhang, Theorem 5.18 in \cite{Zhang2017}) \label{thm:LowerBoundP}
Let $T$ be a tournament with $\pi_2(T)\ge 1$. If $P$ is a non-extendable path in $T$ on $p$ vertices,
then $p \ge 3\pi_2(T)+3$.
\end{theorem}
\begin{proof}
Let $P=u_0u_1\ldots u_{p-1}$. Furthermore, let $F=T-V(P)$. For $u,v\in V(T)$, $u\neq v$,
we denote the set of the intermediate vertices of all $(u,v)$-$2$-paths in $T$ by $V_I(u,v)$.
By the condition of the theorem, $|V_I(u,v)|\geq \pi_2(T) \geq 1$.

Since $P$ is not extendable, any $v\in V(F)$ can only be a dominating, dominated or hybrid vertex of $P$.
Therefore, for any integer $i$ and $j$ such that $0\leq i < j \leq p-1$,
there cannot exist a vertex $v\in V(F)\cap V_I(u_i,u_j)$. Hence, $V_I(u_i,u_j)\subseteq V(P)$.

Since $V_I(u,v)\geq 1$ for all vertex pairs $\{u,v\}$, all arcs are extendable.
Therefore, we may assume $|V(P)|\geq 3$.
Since $|V_I(u_0,u_1)|\geq \pi_2(T)$ and $V_I(u_0, u_1)\subseteq V(P)$,
$u_1$ has at least $\pi_2(T)$ in-neighbors on $P$ besides $u_0$.
So, we have $$|N^-_P(u_1)|\geq \pi_2(T)+1.$$

Consider the set of out-neighbors $N^+_P(u_1)$ of $u_1$ on $P$.
For any $u_i\in N^+_P(u_1)$, we have $i > 1$.
Hence, we have $V_I(u_1, u_i)\subseteq V(P)$, and furthermore, $V_I(u_1, u_i)\subseteq N^+_P(u_1)$.
Since $|V_I(u_1,u_i)|\geq \pi_2(T)$, $u_i$ has at least $\pi_2(T)$ in-neighbors in $N^+_P(u_1)$.
This means that in $\langle N^+_P(u_1) \rangle$ the minimum in-degree is at least $\pi_2(T)$.
Therefore, $|N^+_P(u_1)|(|N^+_P(u_1)|-1)/2 \geq \pi_2(T) |N^+_P(u_1)|$,
from which we get $$|N^+_P(u_1)|\geq 2\pi_2(T)+1.$$

Since $T$ is a tournament, $N^-_P(u_1) \cap N^+_P(u_1) = \emptyset$. We conclude that $p\geq 1+(\pi_2(T)+1)+(2\pi_2(T)+1)=3\pi_2(T)+3$.
\end{proof}

Now we have all the ingredients to prove Theorem~\ref{thm:pi_2(T)i(T)}.

\begin{proof}[\textit{Proof of Theorem~\ref{thm:pi_2(T)i(T)}}]

Let $P=u_0u_1\ldots u_{p-1}$ be a non-extendable path in $T$.
By Theorem~\ref{thm:LowerBoundP}, we have $p \ge 3\pi_2(T)+3 \ge 6$.
Thus we may further let $P^\prime = u_1\ldots u_{p-2}$,
where $|V(P^\prime)|\ge 4$.
We discuss the following two cases.
In Case 1, $P^\prime$ does not have a hybrid vertex.
In Case 2, $P^\prime$ has at least one hybrid vertex,
and thus by Lemma~\ref{lem:h(P)}, $h(P)\leq i(T)+2$.

\medskip
\noindent
\textbf{Case 1. $\boldsymbol{P^\prime}$ does not have a hybrid vertex.}

\noindent
The vertices of $V(T)\setminus V(P)$ can be divided into four sets, i.e.,
those dominating $V(P)$,
those dominating $V(P^\prime)\cup \{u_0\}$ but dominated by $u_{p-1}$,
those dominated by $V(P^\prime)\cup \{u_{p-1}\}$ but dominating $u_{0}$,
and those dominated by $V(P)$, denoted by $N_0^-$, $N_1^-$, $N_1^+$ and
$N_0^+$, respectively. The orders of $N_0^-$, $N_1^-$, $N_1^+$ and
$N_0^+$ are denoted by $n_0^-$, $n_1^-$, $n_1^+$ and $n_0^+$, respectively.
Calculating the sum of the out-degrees of all vertices in $N_0^+\cup N_1^+$,
we have

$$(n_0^++n_1^+)\cdot \frac{n-i(T)-1}{2} \leq \sum_{v\in N_0^+\cup N_1^+}d^+(v)=
n_1^++\frac{(n_0^++n_1^+)(n_0^++n_1^+-1)}{2}+d^+(N_1^+\cup N_0^+, N_0^-\cup N_1^-).$$

Therefore,
$$\frac{n_0^++n_1^+}{2}(n-i(T)-n_0^+-n_1^+)-n_1^+
\leq d^+(N_1^+\cup N_0^+, N_0^-\cup N_1^-)
\leq (n_0^-+n_1^-)(n_0^++n_1^+).$$

\textcolor{black}{If $n_0^++n_1^+=0$, we must have $n_0^-=0$, for there can be no $2$-path from any
$x\in V(P)\setminus \{u_{p-1}\}$ to a $y\in N_0^-$.
Then $n_1^- \ne 0$.
Moreover, there can be only one $2$-path $xu_{p-1}y$ from
any $x\in V(P)\setminus \{u_{p-1}\}$ to a $y\in N_1^-$, thus $\pi_2(T)=1$.
By the condition of the theorem, $2\pi_2(T)-(n+8)/6 > i(T) \ge 0$.
Thus we get $n<4$, contradicting $n\ge 9$.}

Thus we may assume $n_0^++n_1^+\neq 0$ and get
$$n-n_0^+-n_1^+-i(T)-\frac{2n_1^+}{n_0^++n_1^+}\leq 2n_0^-+2n_1^-.$$

Note that $n-n_0^+-n_1^+=p+n_0^-+n_1^-$, so we have
\begin{equation} \label{eqn:CalOutdegree}
p\leq n_0^-+n_1^-+i(T)+\frac{2n_1^+}{n_0^++n_1^+}.
\end{equation}

Similarly, if we consider the sum of the in-degrees of all vertices in $N_0^-\cup N_1^-$,
we obtain
\begin{equation} \label{eqn:CalIndegree}
p\leq n_0^++n_1^++i(T)+\frac{2n_1^-}{n_0^-+n_1^-}.
\end{equation}

Summing up (\ref{eqn:CalOutdegree}) and (\ref{eqn:CalIndegree}), we get
$$2p\leq n-p+2i(T)+\frac{2n_1^+}{n_0^++n_1^+}+\frac{2n_1^-}{n_0^-+n_1^-}.$$

Therefore,
$$3p\leq n+2i(T)+\frac{2n_1^+}{n_0^++n_1^+}+\frac{2n_1^-}{n_0^-+n_1^-}\leq n+2i(T)+4.$$

So,
\begin{equation} \label{eqn:UpperBoundForp}
p\leq \frac{n+2i(T)+4}{3}.
\end{equation}

Next, we consider $2$-paths between vertex pairs in $P^\prime$. By Lemma~\ref{lem:Surplus},

$$\sum_{u,v\in V(P^\prime)}p_2(u,v)= (p-2)(p-3)\pi_2(T)+s(V(P^\prime)) \ge
(p-2)(p-3)\pi_2(T)+{p-2 \choose 2}- \Big\lceil\frac{p-2}{2}\Big\rceil\Big\lfloor\frac{p-2}{2}\Big\rfloor.$$

Since $P^\prime$ does not have a hybrid vertex,
the intermediate vertex of a $2$-path between a vertex pair in $P^\prime$ must be in $V(P)$.
The number of $2$-paths with all vertices in $V(P^\prime)$ is
at most $(p-2)\lfloor (p-3)/2 \rfloor \lceil (p-3)/2 \rceil$.
The number of $2$-paths between a vertex pair in $P^\prime$
that go through $u_0$ or $u_{p-1}$ is at most
$\lfloor (p-2)/2 \rfloor \lceil (p-2)/2 \rceil$.
So,

$$\sum_{u,v\in V(P^\prime)}p_2(u,v) \le (p-2)\Big\lfloor \frac{p-3}{2} \Big\rfloor \Big\lceil \frac{p-3}{2} \Big\rceil
+\Big\lfloor \frac{p-2}{2} \Big\rfloor \Big\lceil \frac{p-2}{2} \Big\rceil.$$

Thus,
\begin{equation} \nonumber
(p-2)(p-3)\pi_2(T)+{p-2 \choose 2}- \Big\lceil\frac{p-2}{2}\Big\rceil\Big\lfloor\frac{p-2}{2}\Big\rfloor
\le (p-2)\Big\lfloor \frac{p-3}{2} \Big\rfloor \Big\lceil \frac{p-3}{2} \Big\rceil
+\Big\lfloor \frac{p-2}{2} \Big\rfloor \Big\lceil \frac{p-2}{2} \Big\rceil,
\end{equation}

from which we obtain
\begin{equation} \label{eqn:pAndPi_2}
\pi_2(T) \le \frac{1}{p-3}\Big\lfloor \frac{p-3}{2} \Big\rfloor \Big\lceil \frac{p-3}{2} \Big\rceil
+\frac{2}{(p-2)(p-3)}\Big\lfloor \frac{p}{2}-1 \Big\rfloor \Big\lceil \frac{p}{2}-1 \Big\rceil - \frac{1}{2}.
\end{equation}

First assume that $p$ is odd. Using (\ref{eqn:pAndPi_2}), we get

$$p\ge 4\pi_2(T)+3-\frac{2}{p-2}.$$

Since $p\ge 6$, clearly $2/(p-2) < 1$. Since $p$ is an odd integer,
$$p\ge 4\pi_2(T)+3.$$

Next assume $p$ is even. Using (\ref{eqn:pAndPi_2}), we get

$$p\ge 4\pi_2(T)+3-\frac{1}{p-3}.$$

Since $p$ is an even integer, similarly to the above
$$p\ge 4\pi_2(T)+4 > 4\pi_2(T)+3.$$

Combining the above with (\ref{eqn:UpperBoundForp}), we have
$$4\pi_2(T)+3\le \frac{n+2i(T)+4}{3},$$

which simplifies to
$$i(T)\ge 6\pi_2(T)-\frac{n-5}{2}.$$

$\newline$

\medskip
\noindent
\textbf{Case 2. $\boldsymbol{P^\prime}$ has a hybrid vertex.}

\noindent
By Lemma~\ref{lem:h(P)}, we have $h(P)\leq i(T)+2$.
Consider the sum of the number of $2$-paths between vertex pairs of $V(P)$. By Lemma~\ref{lem:Surplus}, we have

\begin{equation} \label{eqn:LowerBoundHybrid}
\sum_{u,v\in V(P)}p_2(u,v)= p(p-1)\pi_2(T)+s(V(P))
\ge  p(p-1)\pi_2(T)+{p \choose 2}-\Big\lceil \frac{p}{2} \Big\rceil \Big\lfloor \frac{p}{2} \Big\rfloor.
\end{equation}

All $2$-paths between a vertex pair of $V(P)$ must have their intermediate vertex in $V(P)\cup H(P)$.
The number of those with an intermediate vertex in $V(P)$ is at most
$p \lfloor (p-1)/2 \rfloor \lceil (p-1)/2 \rceil$,
and of those with an intermediate vertex in $H(P)$ is at most
$ \lfloor p/2 \rfloor \lceil p/2 \rceil h(p)$. Thus,

\begin{equation} \label{eqn:UpperBoundHybrid}
\sum_{u,v\in V(P)}p_2(u,v)\leq
p \Big\lfloor \frac{p-1}{2} \Big\rfloor \Big\lceil \frac{p-1}{2} \Big\rceil +
 \Big\lfloor \frac{p}{2} \Big\rfloor \Big\lceil \frac{p}{2} \Big\rceil h(p).
\end{equation}

By (\ref{eqn:LowerBoundHybrid}) and (\ref{eqn:UpperBoundHybrid}),
\begin{equation} \nonumber
p(p-1)\pi_2(T)+{p \choose 2}-\Big\lceil \frac{p}{2} \Big\rceil \Big\lfloor \frac{p}{2} \Big\rfloor \le
p \Big\lfloor \frac{p-1}{2} \Big\rfloor \Big\lceil \frac{p-1}{2} \Big\rceil + \Big\lfloor \frac{p}{2} \Big\rfloor \Big\lceil \frac{p}{2} \Big\rceil h(p),
\end{equation}
that is,
\begin{equation} \label{eqn:Pi_2Hybrid}
\pi_2(T)\le \frac{1}{p-1}\Big\lfloor \frac{p-1}{2} \Big\rfloor \Big\lceil \frac{p-1}{2} \Big\rceil
+\frac{h(p)+1}{p(p-1)}\Big\lfloor \frac{p}{2} \Big\rfloor \Big\lceil \frac{p}{2} \Big\rceil - \frac{1}{2}.
\end{equation}

From (\ref{eqn:Pi_2Hybrid}) we get

\begin{equation} \label{eqn:PAndPi_2Hybrid} \nonumber
p \ge \left\{
\begin{array}
{l l}
4\pi_2(T)-h(P)+2-(h(P)+1)/p, & \mbox{if\ } p \mbox{ is odd,} \\
4\pi_2(T)-h(P)+2-h(P)/(p-1), & \mbox{if\ } p \mbox{ is even.}
\end{array}
\right.
\end{equation}

By Theorem~\ref{thm:LowerBoundP},
$p\ge 3\pi_2(T)+3 > 3(n-9)/12+3 = (n+3)/4$,
and hence $h(P)\le n-p < (3n-3)/4$.
Therefore $(h(P)+1)/p < ((3n-3)/4+1) / ((n+3)/4) = (3n+1)/(n+3)< 3$ and
$h(P)/(p-1)<((3n-3)/4)/ ((n+3)/4-1)= (3n-3)/(n-1)=3$.
Since $p$, $\pi_2(T)$ and $h(P)$ are integers, we have
\begin{equation} \label{eqn:LowerBoundforpWithh(P)}
 p \geq 4\pi_2(T)-h(P)+2-2=4\pi_2(T)-h(P).
\end{equation}

Considering the out-degrees of the vertices of $N_0^+$, we have

$$n_0^+\cdot \frac{n-1-i(T)}{2} \leq \sum_{u\in N_0^+}d^+(u)
\leq \frac{n_0^+(n_0^+-1)}{2}+n_0^+(h(P)+n_0^-), $$
that is,
$$n-n_0^+-n_0^--h(P) \leq n_0^-+i(T)+h(P).$$
Note that $p=n-n_0^+-n_0^--h(P)$, thus we have
\begin{equation} \label{eqn:UpperBoundforPwithn0-}
p\leq n_0^-+i(T)+h(P).
\end{equation}

Similarly, we have
\begin{equation} \label{eqn:UpperBoundforPwithn0+}
p\leq n_0^++i(T)+h(P).
\end{equation}

Summing up (\ref{eqn:UpperBoundforPwithn0-}) and (\ref{eqn:UpperBoundforPwithn0+}), we have
\begin{equation}\nonumber
\begin{split}
2p &\leq n_0^-+n_0^++2i(T)+2h(P) \\
    &= n-p-h(P)+2i(T)+2h(P) \\
    &= n-p+h(P)+2i(T).
\end{split}
\end{equation}

Therefore,
\begin{equation} \label{eqn:UpperBoundforpSum}
3p\leq n+h(P)+2i(T).
\end{equation}

Combining (\ref{eqn:LowerBoundforpWithh(P)}) and (\ref{eqn:UpperBoundforpSum}), we obtain
$$12\pi_2(T)-3h(P) \leq 3p \leq n+h(P)+2i(T).$$
Using that $h(P)\leq i(T)+2$, we get
\begin{equation}  \nonumber
\begin{split}
12\pi_2(T) &\le n+4h(P)+2i(T) \\
 &\le n+4(i(T)+2) + 2i(T) \\
 &= n+6i(T)+8,
\end{split}
\end{equation}
that is,
\begin{equation}\label{i(T)andpi_2(T)} \nonumber
i(T)\ge  2\pi_2(T)-\frac{n+8}{6}.
\end{equation}

Using the condition of the theorem that $\pi_2(T)>(n-9)/12$, we get
$$ 6\pi_2(T)-\frac{n-5}{2}-\Big(2\pi_2(T)-\frac{n+8}{6}\Big)
= 4\pi_2(T)-\frac{4n-46}{12} > \frac{4n-36}{12} -\frac{4n-46}{12}=\frac{5}{6}>0.$$

Hence in either case, $i(T)\ge  2\pi_2(T)-(n+8)/6$, contradicting the condition of the theorem.
\end{proof}

%

\subsection{Proof of Lemma~\ref{lem:i(T)AndPi(T)}}
By (\ref{eqn:i(T)andDegree}), either $\delta^+(T)=(n-1-i(T))/2)$ or $\delta^-(T)=(n-1-i(T))/2)$.
Assume the first equation holds and let $u\in V(T)$ be a vertex in $T$ such that $d^+(u)=\delta^+(T)=(n-1-i(T))/2)$.

Consider the subdigraph $T_u=\langle N^+(u)\rangle$ and any vertex $v$ in it. Every
$\{u,v\}$-$2$-path in $T$ must have its intermediate vertex in $T_u$. Therefore,
for the minimum in-degree $\delta^-_{T_u}$ of $T_u$, we have
$\delta^-_{T_u} \geq \pi_2(T)$.
Thus,
using $|T_u|$ as shorthand for $|V(T_u)|$,
 $$\pi_2(T)|T_u|\leq \sum_{v\in T_u}d_{T_u}^-(v) = |T_u|(|T_u|-1)/2,$$
that is $|T_u|\geq 2\pi_2(T)+1$, that is, $\delta^+(T)\geq 2\pi_2(T)+1$.
Therefore, $2\pi_2(T)+1\leq (n-1-i(T))/2$, we have that
$i(T)\leq n-4\pi_2(T)-3$.

For the case that $\delta^-(T)=(n-1-i(T))/2)$, we can have a similar proof
by considering a vertex $u$ with $d^-(u)=\delta^-(T)=(n-1-i(T))/2)$ and the sum of
the out-degree of the vertex in the subgraph $T_u=\langle N^-(u)\rangle$.

\subsection{Proof of Theorem~\ref{thm:pi_2(T)only}}

Let $T$ be a tournament satisfying the condition of the theorem.
Note that $\pi_2(T) > (7n-10)/36>(n-9)/12$.
Furthermore, by Lemma \ref{lem:i(T)AndPi(T)}, $i(T)\le n-4\pi_2(T)-3$, hence
\begin{equation} \nonumber
\begin{split}
 2\pi_2(T)-\frac{(n+8)}{6} - i(T)
 &\ge  2\pi_2(T)-\frac{(n+8)}{6} - (n-4\pi_2(T)-3) \\
 &= 6\pi_2(T)-\frac{7n-10}{6} \\
 &= 6\Big(\pi_2(T)-\frac{7n-10}{36}\Big) \\
 &> 0.
\end{split}
\end{equation}
Therefore $2\pi_2(T)-(n+8)/6 > i(T)$, and by Theorem \ref{thm:pi_2(T)i(T)},
$T$ is path extendable.

\section{
Random Tournament; Proof of Theorem~\ref{thm:RandomTPi_2} and Theorem~\ref{thm:RandomTPathExt}}
\label{sec:RandomTPi_2}

\subsection{Proof of Theorem~\ref{thm:RandomTPi_2}}
We will need the following Chernoff bound in our proof.

\begin{lemma} \label{lem:Chernoff}
(Multiplicative Chernoff bound, \cite{MitUpf2017}). Suppose $X_1$, $\ldots$, $X_n$ are independent random variables taking values in $\{0, 1\}$,
where $Pr\{X_i=1\}=p_i$.
Let $X$ denote their sum and let $\mu = E[X]$ denote the expectation of $X$. Then for any $ 0<\delta <1$,
$$Pr\{X<(1-\delta)\mu\}<\bigg( \frac{e^{-\delta}}{(1-\delta)^{1-\delta}}\bigg)^\mu.$$
\end{lemma}

\begin{proof}[\textit{Proof of Theorem \ref{thm:RandomTPi_2}}]
Let $u$, $v$ be two vertices in $D \in \overrightarrow{G}(n,p)$. Let the other $n-2$ vertices be denoted by $w_i$,
and define the indicator variable $X_i$ as
\begin{equation} \nonumber
X_i= \left \{
\begin{array}
{l l}
1 & \mbox{if } u \rightarrow w_i \rightarrow v, \\
0 & \mbox{otherwise,}
\end{array}
\right.
\end{equation}
for $ 0 \le i \le n-3$. We clearly have $E[X_i]=Pr\{X_i=1\}=p^2$.

Further define $X=\sum_{i=0}^{n-3}X_i$. Then the variable corresponding to $p_2(u,v)$ is $X$, and $\mu=E[X]=(n-2)p^2$.
Applying the Chernoff bound of Lemma~\ref{lem:Chernoff} to $X$, for $0< \varepsilon < p^2$ as defined in the condition of the theorem, we have
\begin{equation} \nonumber
Pr\{X<(1-\varepsilon)(n-2)p^2\}<\bigg( \frac{e^{-\varepsilon}}{(1-\varepsilon)^{1-\varepsilon}}\bigg)^{(n-2)p^2}.
\end{equation}

Let $N=2(1-\varepsilon)p^2/(\varepsilon(1-p^2))$. For any integer $n>N$, we have
\begin{equation} \nonumber
\begin{split}
(1-\varepsilon)(n-2)p^2-(p^2-\varepsilon)n
&= (1-p^2)\varepsilon n-2 (1-\varepsilon) p^2 \\
&> (1-p^2)\varepsilon N-2 (1-\varepsilon) p^2  \\
&= 0.
\end{split}
\end{equation}
Therefore, $(1-\varepsilon)(n-2)p^2 > (p^2-\varepsilon)n >0$, and
$$Pr\{p_2(u,v)<(p^2-\varepsilon)n\} =
Pr\{X<(p^2-\varepsilon)n\} < Pr\{X<(1-\varepsilon)(n-2)p^2\}
<\bigg( \frac{e^{-\varepsilon}}{(1-\varepsilon)^{1-\varepsilon}}\bigg)^{(n-2)p^2}.$$

Consider the event $\pi_2(D) < (p^2-\varepsilon)n$,
which happens if and only if there is at least a vertex pair $\{u, v\}$ in $D$
such that $p_2(u,v) < (p^2-\varepsilon)n$.
This event is the union of the events $p_2(u,v) < (p^2-\varepsilon)n$,
where $\{u, v\}$ runs over all vertex pairs in $D$.
Applying Boole's inequality that the probability of the union of events
is at most the sum of the probabilities of every event, we get
\begin{equation} \nonumber
\begin{split}
Pr\{\pi_2(D)<(p^2-\varepsilon)n \}
&= Pr\Big\{\bigcup_{\substack{u,v\in V(D) \\ u\ne v}}p_2(u,v)<(p^2-\varepsilon)n \Big\} \\
&\le \sum_{\substack{u,v\in V(D) \\ u\ne v}} Pr\{p_2(u,v)<(p^2-\varepsilon)n \}\\
&< n(n-1) \bigg( \frac{e^{-\varepsilon}}{(1-\varepsilon)^{1-\varepsilon}}\bigg)^{(n-2)p^2}.
\end{split}
\end{equation}

Let $g(x)=e^{-x}/(1-x)^{1-x}$, $0\le x \le 1/4$. Then $g'(x)=e^{-x}(1-x)^{x-1}\log (1-x)$.
Since $g'(x)<0$ for $0<x\le 1/4$, $g(x)$ is strictly decreasing in $(0,1/4)$, and $g(x)<g(0)=1$.
Therefore, $ e^{-\varepsilon}/(1-\varepsilon)^{1-\varepsilon} <1$, and for a fixed $p>0$,
$$ \lim_{n\rightarrow \infty}  n(n-1)\bigg( \frac{e^{-\varepsilon}}{(1-\varepsilon)^{1-\varepsilon}}\bigg)^{(n-2)p^2}=0,$$
yielding
$$\lim_{n\rightarrow \infty} Pr\{\pi_2(D)\ge (p^2-\varepsilon)n \}
=1-\lim_{n\rightarrow \infty} Pr\{\pi_2(D)< (p^2-\varepsilon)n \}
=1.$$
We conclude that almost all oriented graphs $D\in \overrightarrow{G}(n,p)$ satisfy $\pi_2(D)\ge (p^2-\varepsilon)n$.
\end{proof}


\subsection{Proof of Theorem \ref{thm:RandomTPathExt}}

Letting $p=1/2$ in $\overrightarrow{G}(n,p)$, we
obtain
a model for random tournaments.
By Theorem~\ref{thm:RandomTPi_2}, for $0 < \varepsilon <1/4$, almost all tournaments $T$ satisfy
$\pi_2(T)\ge (1/4-\varepsilon)n$. 
Letting $\varepsilon < 1/18$, we 
get
$(1/4- \varepsilon)n > 7n/36 > (7n-10)/36$. 
Hence, by Theorem~\ref{thm:pi_2(T)only},
almost all tournaments are path extendable.

\section{Conclusion and open problems} \label{sec:Final}

Earlier results on tournaments have revealed that in tournaments strongness implies hamiltonicity, pancyclicity
and cycle extendability (with an exceptional class), but not path extendability.
However, by considering analogous structural results on hamiltonian properties of undirected graphs, we were tempted to expect that imposing similar conditions as in
Theorem~\ref{thm:irregularitystronglypanconnected}
and Theorem~\ref{thm:2pathallpaths} on a tournament would imply that it is path extendable.
Our results for path extendability show that the situation is in fact quite different.
Although path extendability seems to be
asking only a little bit more than strong panconnectedness
and completely strong path-connectedness,
to guarantee path extendability we need to impose significantly 
stronger conditions,
in particular 
conditions that
combine restrictions on $i(T)$ and $\pi_2(T)$.
%
Based on our work, we leave the following two problems 
for future research.


Lemma \ref{lem:i(T)AndPi(T)} shows that
a larger $\pi_2(T)$ implies a smaller $i(T)$. In particular,
if $\pi_2(T)$ attains its supremum $(n-3)/4$, we get $i(T)=0$, and $T$ is actually doubly regular.
In this sense, the result is tight.
It
is
worth noting that the reverse does not hold, that is,
a tournament $T$ with a small $i(T)$ need not
have %
a
large $\pi_2(T)$.
Consider any tournament $T$ on $n=2k+1$ vertices
for some integer
$k\ge 3$, 
constructed as follows (see Figure~\ref{figure:RegularNoTwoPath}).
\begin{figure}[h]
\centering
\includegraphics[width=0.28\linewidth]{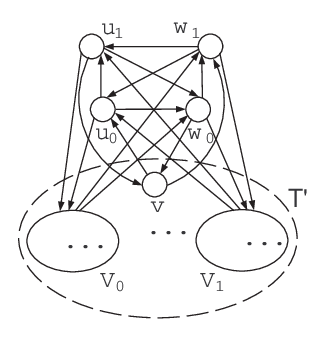}
\caption{A regular tournament without $(u_0,u_1)$-$2$-paths}
\label{figure:RegularNoTwoPath}
\end{figure}
Take a regular tournament $T^\prime$ on $2k-3$ vertices.
Let $V(T^\prime)=V_0 \cup V_1 \cup \{v\}$ be a partition of $V(T^\prime)$ with $|V_0|=|V_1|=k-2$.
$T$ is constructed by adding four new vertices $u_0$, $u_1$, $w_0$ and $w_1$ to $T^\prime$,
and adding arcs so that $V_1 \rightarrow \{u_0, u_1\} \rightarrow V_0$,
$V_0 \rightarrow \{w_0, w_1\} \rightarrow V_1$,
$w_1 \rightarrow \{u_0, u_1\} \rightarrow w_0$,
$u_0 \rightarrow u_1 \rightarrow v \rightarrow u_0$,
$w_0 \rightarrow v \rightarrow w_1$ and $w_0\rightarrow w_1$.
It can be verified that $T$ is regular but there is no $(u_0,u_1)$-$2$-path in $T$.

It may be interesting to characterize the class of tournaments for which equality holds in
Lemma~\ref{lem:i(T)AndPi(T)}, which includes all doubly regular tournaments.
\begin{problem} Characterize the class of tournaments $T$ on $n$ vertices with $i(T)=n-4\pi_2(T)-3$. \end{problem}

If we let $i(T)=0$ in Theorem~\ref{thm:pi_2(T)i(T)},
then $T$ is a regular tournament and the condition becomes $\pi_2(T)>(n+8)/12$.
There is a clear gap between this bound and the sharp one $(n-9)/12$ of Theorem~\ref{thm:2PathsPathExt}.
Thus, there seems to be some space to further tighten the conditions in
Theorem~\ref{thm:pi_2(T)i(T)} and Theorem~\ref{thm:pi_2(T)only}.
However, we believe such improvements will require a different approach with a much more subtle and detailed structural analysis.

\begin{problem} Improve the bound $i(T)< 2\pi_2(T)-(n+8)/6$ for path extendability
of a tournament $T$ on $n$ vertices in Theorem~\ref{thm:pi_2(T)i(T)}. \end{problem}

\end{document}